\documentclass[12pt]{article}

\usepackage[margin=1in]{geometry}
\usepackage{amsmath,amsthm,amssymb,amsfonts}
\usepackage{mathtools}
\usepackage{hyperref}
\usepackage{url}
\usepackage[numbers,sort&compress]{natbib}
\usepackage{enumitem}
\usepackage{silence}
\usepackage[expansion=false]{microtype}
\usepackage{booktabs}
\usepackage{graphicx}

\theoremstyle{plain}
\newtheorem{theorem}{Theorem}[section]
\newtheorem{proposition}[theorem]{Proposition}
\newtheorem{lemma}[theorem]{Lemma}
\newtheorem{corollary}[theorem]{Corollary}

\theoremstyle{definition}
\newtheorem{definition}[theorem]{Definition}
\newtheorem{example}[theorem]{Example}
\newtheorem{hypothesis}{Hypothesis}

\theoremstyle{remark}
\newtheorem{remark}[theorem]{Remark}

\newcommand{\R}{\mathbb{R}}
\newcommand{\C}{\mathbb{C}}
\newcommand{\tr}{\operatorname{tr}}
\newcommand{\doi}[1]{\textsc{doi:}~\texttt{#1}}
\DeclareMathOperator{\dt}{det}
\DeclareMathOperator{\spec}{spec}
\DeclareMathOperator{\sgn}{sgn}
\DeclareMathOperator{\Real}{Re}
\DeclareMathOperator{\Imag}{Im}
\newcommand{\ip}[2]{\langle #1,\, #2\rangle}
\newcommand{\Pstar}{P^{*}}
\newcommand{\xstar}{x^{*}}
\newcommand{\ystar}{y^{*}}
\newcommand{\wO}{\omega_{0}}
\newcommand{\lone}{\ell_{1}}

\title{The trace field on the prey nullcline:\\
	criticality of planar Hopf bifurcations\\
	on the admissible branch}

\author{%
	E.~Chan-L\'opez\thanks{%
		E-mail: \texttt{eduardo.clopez13@gmail.com}.\;
		\textsc{orcid}:~0009-0003-7712-7907.}
}

\date{%
	Divisi\'on Acad\'emica de Ciencias B\'asicas,
	Universidad Ju\'arez Aut\'onoma de Tabasco, 86690 Cunduac\'an,
	Tabasco, Mexico}

\begin{document}
	\maketitle
	
	\begin{abstract}
		Along the prey nullcline $y=g(x)$ of a planar system, the Jacobian entry
		satisfies the identity $J_{11}=P\,g'$, where
		$P=-f_{1y}|_{\gamma}$. This identity holds for any planar field whose first
		nullcline is a graph with $f_{1y}\neq0$ on it. Under the positivity condition
		$P>0$ and predator self-damping, the Hopf trace condition therefore forces
		$g'(\xstar)>0$: critical points of $g$ act as spectral barriers, and the Hopf
		locus lies on ascending branches. This geometric localization determines
		where an oscillatory instability can occur. The present paper determines how
		the bifurcation unfolds once such a point is located.
		
		We work with the trace field $D=f_{1x}+f_{2y}$ and its restriction
		$\tau=D|_{\gamma}$ to the nullcline. Using the nullcline as a coordinate
		$(u=x,\ v=y-g(x))$ and Hadamard's lemma, the field becomes
		$\dot u=-vH$, $\dot v=W$ with $H(u,0)=P$. The resulting mixed jet of $W$ is
		determined by the jets of $\tau$ and $\nu=f_{2}|_{\gamma}$, yielding a closed
		formula for the first Lyapunov coefficient $\lone$ in trace-field form with
		five terms. The formula applies to predator-dependent functional responses.
		
		Two structural reductions are established in full generality, without a
		Gause hypothesis. First, the pairing of the cubic form with the critical
		eigenvectors is purely imaginary in the two extreme slots, so that
		$\partial\lone/\partial f_{1yyy}=0$ and
		$\partial\lone/\partial f_{2xxx}=0$. Second, straightening the nullcline
		annihilates two further slots. Within the Gause class, a third and distinct
		mechanism, the reparametrization $\varphi\mapsto g$, removes $p'''$. These
		three mechanisms have different scopes and are proved separately.
		
		Comparing $\lone$ with the $y$-affine truncation of the $3$-jet at fixed
		linear part gives
		$\lone=\lone^{\mathrm{aff}}+\Delta$, where
		$\Delta=\Lambda\bigl(f_{1xyy}+\wO^{-2}
		\ip{\partial_{x}f}{\nabla D}f_{1yy}\bigr)$ and $\Lambda>0$ under the Hopf
		hypotheses alone. On the Hopf locus,
		$\ip{\partial_{x}f}{\nabla D}
		=J_{11}\tau'+(\wO^{2}/P)D_y$.
		Thus the same quantity $J_{11}=Pg'$ that localizes the bifurcation at first
		order reappears at third order as the weight of the tangential trace
		derivative. The bridge identity contains no eigenvector and is unaffected
		by eigenvector normalization, whereas $\lone$, $\lone^{\mathrm{aff}}$,
		$\Delta$, and $\Lambda$ scale by $|c|^2$ under $q\mapsto cq$; their signs
		and ratios are unchanged.
	\end{abstract}
	
	
	
	\section{Introduction}\label{sec:intro}
	Hopf bifurcation provides a local mechanism through which an equilibrium of a
	planar dynamical system loses stability and a small-amplitude periodic orbit is
	created or destroyed. In applications, however, locating a Hopf bifurcation and
	determining its criticality are distinct problems. The linearization determines
	where the equilibrium becomes spectrally critical, whereas the first Lyapunov
	coefficient determines how the nonlinear dynamics unfold there. In particular,
	the sign of the first Lyapunov coefficient distinguishes the local supercritical
	and subcritical scenarios under the usual transversality assumptions
	\cite{Kuznetsov2004}. Thus, a geometric criterion that localizes the Hopf locus
	does not by itself determine the type of oscillatory instability that occurs on
	that locus.
	
	This distinction is especially relevant in planar predator--prey systems.
	For a broad class of models, the prey nullcline provides a natural geometric
	organizer of the positive equilibrium set; this is the classical reading of the
	hump-shaped prey nullcline introduced by Rosenzweig and MacArthur
	\cite{Rosenzweig1963}, later sharpened into the paradox of enrichment
	\cite{Rosenzweig1971}. In Gause-type systems \cite{KuangFreedman1988}, the
	relation between the slope of the prey nullcline and the $(1,1)$ entry of the
	Jacobian makes it possible to constrain the location of Hopf bifurcations
	without computing the full nonlinear normal form. In previous work, this
	observation was used to show that, under a positivity condition on the predation
	factor and a predator self-damping condition, Hopf bifurcations can occur only
	on ascending branches of the prey nullcline \cite{Companion}. The
	critical points of the nullcline therefore act as spectral barriers: they
	separate regions in which an oscillatory instability cannot occur. That result
	answers a geometric localization question, but it leaves open the next local
	question: once a Hopf point has been confined to an admissible branch, what
	determines its criticality?
	
	The standard answer is the first Lyapunov coefficient. In general, however,
	$\ell_1$ is expressed in terms of the quadratic and cubic jets of the vector
	field and of the associated eigenvector and resolvent data
	\cite[\S3.5]{Kuznetsov2004}, and it is in this form that it is evaluated by
	numerical bifurcation software \cite{DhoogeGovaertsKuznetsov2003}. Such a
	formula is well suited to a general normal-form calculation, but it obscures the
	geometry that is available in a system possessing a graph nullcline. In
	particular, the full third-order jet contains many coefficients, and it is not
	immediately clear which of them actually influence the criticality of the Hopf
	bifurcation. A direct substitution of a particular predator--prey model into the
	standard formula can therefore produce a complicated expression without
	revealing which features of the model are dynamically essential; detailed
	bifurcation analyses of specific families, such as the study of a nonmonotonic
	response in \cite{XiaoRuan2001}, obtain the criticality but at the cost of a
	computation tied to that family.
	
	The purpose of this paper is to reorganize the first Lyapunov coefficient
	around the geometry of the prey nullcline and the trace field. Let
	\[
	D=f_{1x}+f_{2y}
	\]
	denote the trace field and let
	\[
	\tau=D|_{\gamma},
	\qquad
	\nu=f_2|_{\gamma}
	\]
	denote its restriction and the predator component restricted to the prey
	nullcline, respectively. If the first nullcline is locally written as
	\[
	\gamma=\{(x,g(x))\},
	\]
	and
	\[
	P=-f_{1y}|_{\gamma}>0,
	\]
	then differentiation of the nullcline identity gives the elementary but
	structurally important relation
	\[
	J_{11}=Pg'.
	\]
	Unlike its usual formulation in Gause systems, this identity requires no
	specific ecological decomposition of the vector field: it follows solely from
	the graph structure of the nullcline. Consequently, the same quantity that
	controls the spectral localization of the Hopf point also provides the natural
	tangential coefficient for the local nonlinear analysis.
	
	Using the nullcline itself as a coordinate, we derive a closed expression for
	the first Lyapunov coefficient in terms of a small collection of geometric
	jet data (Theorem~\ref{thm:main}). In its trace-field form the resulting
	formula contains five terms and separates tangential and transverse
	derivatives of the trace field. In this representation,
	criticality is no longer presented as an opaque function of the complete
	third-order Taylor expansion. Instead, the relevant information is organized
	according to its relation to the nullcline and to the trace field. The formula
	is valid for predator-dependent functional responses and therefore applies
	beyond the predator-independent Gause setting considered in the localization
	result.
	
	The reduction of the relevant jet is itself a structural result. We show first
	that, for an arbitrary planar vector field at a Hopf point, two apparently
	independent cubic coefficients,
	\[
	f_{1yyy}
	\qquad\text{and}\qquad
	f_{2xxx},
	\]
	do not contribute to $\ell_1$ (Proposition~\ref{prop:generalelision}). This
	cancellation does not rely on a nullcline,
	on straightening coordinates, or on the Gause structure. It follows from the
	reality properties of the critical eigenvector pairing. Thus, these two
	directions of the cubic jet are dynamically invisible to the first Lyapunov
	coefficient.
	
	Two additional reductions arise from different geometric mechanisms. After
	straightening the nullcline, two further cubic slots are annihilated by the
	same local structure that produces the adapted coordinate representation.
	Within the Gause class, a third and independent reduction occurs when the prey
	growth law is reparametrized in terms of the nullcline function $g$: the
	third derivative $p'''$ is eliminated from the expression for $\ell_1$. These
	three mechanisms have different origins and different scopes. Keeping them
	separate is essential: the first is a general property of planar Hopf
	bifurcation, the second is associated with the nullcline geometry, and the
	third is specific to the Gause parametrization.
	
	The predator-dependent case reveals a further layer of structure. Functional
	responses that depend on the predator density, such as those of Beddington and
	DeAngelis \cite{Beddington1975,DeAngelis1975} and of Crowley and Martin
	\cite{CrowleyMartin1989}, are not a mathematical convenience: in a comparative
	study of nineteen data sets they were found to describe observed feeding rates
	better than the prey-dependent Holling type~II model \cite{SkalskiGilliam2001},
	and the degree of predator dependence ranges from weak interference to full
	ratio dependence \cite{ArditiGinzburg1989}. We compare the full coefficient
	$\ell_1$ with the coefficient $\ell_1^{\mathrm{aff}}$ obtained from the
	$y$-affine truncation of the third-order jet while keeping the linear part
	fixed. This gives the exact decomposition (Theorem~\ref{thm:decomp})
	\[
	\ell_1=\ell_1^{\mathrm{aff}}+\Delta,
	\]
	where
	\[
	\Delta
	=
	\Lambda
	\left(
	f_{1xyy}
	+
	\frac{\langle \partial_x f,\nabla D\rangle}{\omega_0^2}
	f_{1yy}
	\right),
	\qquad
	\Lambda>0.
	\]
	This decomposition identifies precisely the two channels through which
	predator dependence can affect the first Lyapunov coefficient. It also makes
	clear why replacing a predator-dependent response by a frozen response is not
	the appropriate comparison: such a replacement can alter the linear part and
	therefore changes the Hopf frequency itself. The relevant comparison is a jet
	truncation performed at fixed linearization.
	
	The quantity
	\[
	\langle \partial_x f,\nabla D\rangle
	\]
	appearing in the predator-dependent correction is not an accidental
	combination of derivatives. On the Hopf locus, we prove the identity
	(Corollary~\ref{cor:bridge})
	\[
	\langle \partial_x f,\nabla D\rangle
	=
	J_{11}\tau'
	+
	\frac{\omega_0^2}{P}D_y.
	\]
	Using $J_{11}=Pg'$ then yields
	\[
	\Delta
	=
	\Lambda
	\left[
	f_{1xyy}
	+
	\left(
	\frac{Pg'}{\omega_0^2}\tau'
	+
	\frac{D_y}{P}
	\right)f_{1yy}
	\right].
	\]
	This identity provides the main conceptual link between the localization and
	criticality problems. At first order, $J_{11}=Pg'$ determines on which
	branches a Hopf bifurcation can occur. At third order, the same quantity
	weights the tangential derivative of the trace field in the correction
	produced by predator dependence. Thus, spectral localization and nonlinear
	criticality are not separate calculations placed side by side: they are
	successive orders of the same local geometric organization.
	
	The resulting framework also clarifies the meaning of the affine comparison.
	The quantity $\ell_1^{\mathrm{aff}}$ is not the ``nullcline contribution'' to
	$\ell_1$; it is the first Lyapunov coefficient of a specific $y$-affine jet
	truncation at fixed linear part. The correction $\Delta$ isolates exactly the
	two cubic channels associated with predator-dependent curvature. In particular,
	the correction need not be small. In the examples considered below, it
	preserves the sign of $\ell_1$ but changes its magnitude substantially in the
	Beddington--DeAngelis case, illustrating that predator dependence can affect
	the quantitative nonlinear scaling even when the qualitative criticality
	remains unchanged.
	
	These results lead to a local hierarchy for planar predator--prey Hopf
	bifurcations:
	\[
	\text{nullcline geometry}
	\longrightarrow
	\text{spectral localization}
	\longrightarrow
	\text{jet reduction}
	\longrightarrow
	\text{Hopf criticality}.
	\]
	The hierarchy is useful for both analysis and computation. Rather than
	evaluating the full normal-form expression independently at every model
	instance, one can first identify the admissible branch from the nullcline
	geometry and then evaluate the reduced criticality formula on that branch.
	The resulting organization separates universal planar properties from
	features associated with nullcline geometry and from genuinely
	predator-dependent effects.
	
	The paper is organized as follows. Section~2 introduces the graph-nullcline
	framework, the Hopf hypotheses, the normalization of the critical
	eigenvectors, and the standard expression for the first Lyapunov coefficient.
	Section~3 establishes the admissible-branch localization result in the present
	general graph-nullcline setting. Section~4 introduces the nullcline-adapted
	coordinates and derives the jet relations used in the criticality calculation.
	Section~5 identifies the third-order coefficients that do not contribute to
	$\ell_1$ and separates the corresponding elimination mechanisms. Section~6
	derives the closed formula for $\ell_1$, which in its trace-field form
	contains five terms. Section~7 specializes the
	result to Gause systems and establishes the reparametrization reduction.
	Section~8 analyzes predator-dependent functional responses and derives the
	decomposition $\ell_1=\ell_1^{\mathrm{aff}}+\Delta$ together with its bridge
	to the localization identity. Section~9 gives exact and high-precision
	validations for representative predator--prey models. Section~10 discusses
	the resulting geometric organization, Bautin candidates, and the limits of
	the present planar framework.
	
	\section{Framework}\label{sec:framework}
	
	All fields are $C^{4}$ on an open $U\subseteq\R^{2}$. We write $A=Df(\Pstar)$
	with entries $J_{ij}$ and $D=f_{1x}+f_{2y}$.
	
	\begin{definition}[Graph nullcline]\label{def:graph}
		$f$ has a \emph{graph nullcline} over an open interval $I$ if there is
		$g\in C^{4}(I)$ with $(x,g(x))\in U$, $f_{1}(x,g(x))=0$ and
		$f_{1y}(x,g(x))\neq0$ for $x\in I$. We put
		$\gamma=\{(x,g(x)):x\in I\}$ and $P(x)=-f_{1y}(x,g(x))$.
	\end{definition}
	
	\begin{hypothesis}\label{HGraph} $P>0$ on $I$. \end{hypothesis}
	
	\begin{lemma}[Nullcline trace identity]\label{lem:traceid}
		Under Definition~\ref{def:graph}, $f_{1x}(x,g(x))=P(x)g'(x)$ for all $x\in I$.
		Under \ref{HGraph}, $\sgn f_{1x}|_{\gamma}=\sgn g'$, and
		$\tau:=D|_{\gamma}=Pg'+f_{2y}|_{\gamma}$.
	\end{lemma}
	
	\begin{proof}
		$x\mapsto f_{1}(x,g(x))$ vanishes identically on $I$ and is $C^{4}$, so
		$f_{1x}+f_{1y}g'=0$ on $\gamma$, i.e.\ $f_{1x}=-f_{1y}g'=Pg'$.
	\end{proof}
	
	\begin{remark}[Relation to the companion identity]\label{rem:companionid}
		For a Gause field $P=p$ and Lemma~\ref{lem:traceid} is
		\cite{Companion}. For a predator-dependent response
		$f_{1}=x\varphi(x)-y\Phi(x,y)$ one has $f_{1y}=-(\Phi+y\Phi_{y})$, hence
		\begin{equation}\label{eq:bracket}
			P=-f_{1y}\big|_{\gamma}=\Phi+g\,\Phi_{y},
		\end{equation}
		so Lemma~\ref{lem:traceid} is the form of the identity recorded in
		\cite{Companion} freed of the Gause decomposition: no structure of $f_{1}$
		enters, only Definition~\ref{def:graph}. Hypothesis~\ref{HGraph} is the
		positivity of the bracket required there. The identity is not free of
		hypotheses; it is free of hypotheses on the \emph{form} of $f$.
	\end{remark}
	
	For an equilibrium $\Pstar$ we impose
	\begin{hypothesis}\label{HTrace} $\tr A=0$; \end{hypothesis}
	\begin{hypothesis}\label{HDet} $\dt A=\wO^{2}>0$; \end{hypothesis}
	\begin{hypothesis}\label{HTrans} the eigenvalue pair crosses the imaginary axis
		transversally in the chosen parameter. \end{hypothesis}
	Under \ref{HTrace}--\ref{HDet}, $\spec A=\{\pm i\wO\}$, $J_{22}=-J_{11}$,
	$J_{12}\neq0$ and
	\begin{equation}\label{eq:hopfid}
		J_{11}^{2}+\wO^{2}=-J_{12}J_{21}>0 .
	\end{equation}
	Hypothesis \ref{HTrans} enters only in the interpretation of $\sgn\lone$ as
	criticality of the bifurcating cycle; no algebraic statement below uses it.
	
	Let $q,p\in\C^{2}$ satisfy $Aq=i\wO q$, $A^{\!\top}p=-i\wO p$,
	$\ip{p}{q}=\bar p^{\top}q=1$, let $B,C$ be the second and third derivative
	forms of $f$ at $\Pstar$, and set
	\begin{equation}\label{eq:kuz}
		\lone=\frac{1}{2\wO}\Real\Bigl[\ip{p}{C(q,q,\bar q)}
		-2\ip{p}{B(q,A^{-1}B(q,\bar q))}
		+\ip{p}{B(\bar q,(2i\wO I-A)^{-1}B(q,q))}\Bigr].
	\end{equation}
	
	\begin{lemma}[Scaling]\label{lem:normalization}
		If $q$ is replaced by $cq$, $c\in\C^{\times}$, and $p$ by $p/\bar c$, then the
		right-hand side of \eqref{eq:kuz} is multiplied by $|c|^{2}$. Its sign is
		therefore independent of the choice; its value is not.
	\end{lemma}
	
	\begin{proof}
		With $z=\ip{p}{w}$ one has $w=zq+\bar z\bar q$, and on the centre manifold the
		normal form is $\dot z=i\wO z+c_{1}z|z|^{2}+O(|z|^{4})$ with \eqref{eq:kuz}
		equal to $\Real(c_{1})/\wO$ \cite[\S3.5]{Kuznetsov2004}. Replacing $q$ by $cq$
		replaces $z$ by $\zeta=z/c$, and substituting gives
		$\dot\zeta=i\wO\zeta+|c|^{2}c_{1}\zeta|\zeta|^{2}+O(|\zeta|^{4})$.
	\end{proof}
	
	\begin{definition}[Prey normalization]\label{def:norm}
		Throughout,
		\begin{equation}\label{eq:qp}
			q=\Bigl(1,\ \frac{i\wO-J_{11}}{J_{12}}\Bigr),
			\qquad
			\bar p=\frac{1}{\sigma}\Bigl(1,\ -\frac{J_{11}-i\wO}{J_{21}}\Bigr),
			\quad
			\sigma=1-\frac{(J_{11}-i\wO)^{2}}{J_{12}J_{21}}=\frac{2i\wO}{J_{11}+i\wO}.
		\end{equation}
		All values of $\lone,\lone^{\mathrm{aff}},\Delta,\Lambda$ refer to
		\eqref{eq:qp}. We do not call \eqref{eq:kuz} invariant.
	\end{definition}
	
	\begin{remark}[The two conventions]\label{rem:conventions}
		The choice $q=(J_{12},i\wO-J_{11})$ also occurs; by
		Lemma~\ref{lem:normalization} the two differ by $|J_{12}|^{2}$. Ratios computed
		within one convention, such as $\Delta/\lone$, are unaffected. Conversions are
		flagged where numbers from the other convention are quoted.
	\end{remark}
	
	\section{The admissible branch}\label{sec:localization}
	
	\begin{hypothesis}\label{HDamp} $f_{2y}(\Pstar)<0$. \end{hypothesis}
	
	\begin{remark}\label{rem:dampingsame}
		For a Gause field $f_{2}=yQ$ one has $f_{2y}=Q+yQ_{y}$, which at a coexistence
		equilibrium reduces to $\ystar Q_{y}(\Pstar)$. Hypothesis~\ref{HDamp} is
		therefore the predator self-damping hypothesis of \cite{Companion} evaluated at
		$\Pstar$.
	\end{remark}
	
	\begin{theorem}[Localization; general form of {\cite{Companion}}]
		\label{thm:localization}
		Let $f$ have a graph nullcline over $I$ and satisfy \ref{HGraph}. If
		$\Pstar=(\xstar,g(\xstar))$ is an equilibrium satisfying \ref{HTrace} and
		\ref{HDamp}, then $g'(\xstar)>0$. In particular no such equilibrium occurs at a
		critical point of $g$.
	\end{theorem}
	
	\begin{proof}
		By Lemma~\ref{lem:traceid}, $\tau(\xstar)=P(\xstar)g'(\xstar)+f_{2y}(\Pstar)$,
		and \ref{HTrace} says $\tau(\xstar)=0$; hence
		$P(\xstar)g'(\xstar)=-f_{2y}(\Pstar)>0$ by \ref{HDamp}, and $g'(\xstar)>0$ by
		\ref{HGraph}. If $g'(\xstar)=0$ then $\tau(\xstar)=f_{2y}(\Pstar)<0$,
		contradicting \ref{HTrace}.
	\end{proof}
	
	\begin{remark}[What ``spectral barrier'' means here]\label{rem:barrier}
		The phrase abbreviates the exclusion just proved and nothing more: at a
		critical point of $g$ one has $J_{11}=0$ by Lemma~\ref{lem:traceid}, hence
		$\tr A=f_{2y}<0$ under \ref{HDamp}, so the Hopf trace condition
		\ref{HTrace} cannot hold there. No statement is made about the dynamics at
		such a point, about other bifurcations, or about the global structure of the
		Hopf set.
	\end{remark}
	
	\begin{remark}\label{rem:scope}
		Only \ref{HTrace} is used, so the conclusion applies to any equilibrium with
		vanishing trace, including Bogdanov--Takens points, a case also recorded in
		\cite{Companion}. If $f_{2y}(\Pstar)=0$ the argument degenerates to
		$g'(\xstar)=0$, which is the classical Rosenzweig--MacArthur situation
		\cite{Rosenzweig1963}; \cite{Companion} discusses this limit and the Gause
		specializations in detail, and we do not repeat them.
	\end{remark}
	
	\section{Straightening the nullcline}\label{sec:normalform}
	
	\begin{lemma}[Hadamard]\label{lem:hadamard}
		Let $V$ be open with $V\cap(\{u\}\times\R)$ an interval containing $0$, and
		$F\in C^{m}(V)$, $m\ge1$, with $F(u,0)=0$. Then $F(u,v)=v\Phi(u,v)$ with
		$\Phi\in C^{m-1}(V)$ and $\Phi(u,0)=\partial_{v}F(u,0)$.
	\end{lemma}
	
	\begin{proof}
		$F(u,v)=\int_{0}^{1}\frac{d}{dt}F(u,tv)dt=v\int_{0}^{1}\partial_{v}F(u,tv)dt$;
		set $\Phi(u,v)=\int_{0}^{1}\partial_{v}F(u,tv)dt$ and differentiate under the
		integral sign.
	\end{proof}
	
	\begin{theorem}[Nullcline normal form]\label{thm:normalform}
		Under Definition~\ref{def:graph}, in the coordinates $u=x$, $v=y-g(x)$,
		\begin{equation}\label{eq:normalform}
			\dot u=-v\,H(u,v),\qquad \dot v=W(u,v),
		\end{equation}
		with $H\in C^{3}$, $W\in C^{3}$, $H(u,0)=P(u)$ and
		$W(u,v)=f_{2}(u,v+g(u))-g'(u)f_{1}(u,v+g(u))$. The coordinate change is a
		$C^{4}$ shear with unit Jacobian determinant, and it maps the vector $q$ of
		\eqref{eq:qp} to a vector with first component $1$; by
		Lemma~\ref{lem:normalization} the value of \eqref{eq:kuz} is unchanged.
		The regularity $W\in C^{3}$ is the exact one produced by the construction,
		since $W$ contains the factor $g'\in C^{3}$; it is also all that is needed,
		because \eqref{eq:kuz} uses derivatives of the transformed field only up to
		order three.
	\end{theorem}
	
	\begin{proof}
		Put $\tilde f_{1}(u,v)=f_{1}(u,v+g(u))$, so $\tilde f_{1}(u,0)=0$;
		Lemma~\ref{lem:hadamard} gives $\tilde f_{1}=-vH$ with
		$H(u,0)=-f_{1y}(u,g(u))=P(u)$. Then $\dot u=f_{1}=-vH$ and
		$\dot v=f_{2}-g'f_{1}$. The shear has matrix
		$\left(\begin{smallmatrix}1&0\\-g'&1\end{smallmatrix}\right)$, of determinant
		$1$, and preserves first components.
	\end{proof}
	
	\begin{proposition}[Jet identities]\label{prop:jets}
		With $\nu=f_{2}|_{\gamma}$ and $\tau=D|_{\gamma}$, for $0\le k\le3$,
		\[
		\partial_{u}^{k}W(u,0)=\nu^{(k)}(u),\qquad
		\frac{d^{k}}{du^{k}}\bigl[\partial_{v}W(u,0)\bigr]=\tau^{(k)}(u).
		\]
	\end{proposition}
	
	\begin{proof}
		$W(u,0)=f_{2}(u,g(u))=\nu(u)$. Differentiating in $v$,
		$\partial_{v}W=f_{2y}(u,v+g)-g'f_{1y}(u,v+g)$, which at $v=0$ equals
		$f_{2y}+g'P=f_{2y}+f_{1x}=\tau(u)$ by Lemma~\ref{lem:traceid}. In both
		identities the restriction to $v=0$ is taken first, so that each side is a
		function of $u$ alone; differentiating those functions $k$ times gives the
		statement. No mixed partial derivative $\partial_{u}^{k}\partial_{v}W$ of
		order four is claimed, and none is needed: only $k\le2$ is used below, in
		$W_{11}=\tau'$ and $W_{21}=\tau''$.
	\end{proof}
	
	\begin{proposition}[Hopf data]\label{prop:hopfdata}
		$(\xstar,0)$ is an equilibrium of \eqref{eq:normalform} iff $\nu(\xstar)=0$,
		with linearization
		$\left(\begin{smallmatrix}0&-P\\ \nu'&\tau\end{smallmatrix}\right)$. Hypotheses
		\ref{HTrace}--\ref{HDet} hold iff $\nu(\xstar)=\tau(\xstar)=0$ and
		$\wO^{2}=P\nu'(\xstar)>0$, in which case
		\begin{equation}\label{eq:A0}
			A_{0}=\begin{pmatrix}0&-P\\ \wO^{2}/P&0\end{pmatrix}.
		\end{equation}
	\end{proposition}
	
	\begin{proof}
		$\partial_{u}\dot u=-v\partial_{u}H=0$ and $\partial_{v}\dot u=-P$ at $v=0$;
		the second row is Proposition~\ref{prop:jets} for $k=0,1$. A real $2\times2$
		matrix has spectrum $\{\pm i\wO\}$, $\wO>0$, iff its trace vanishes and its
		determinant is $\wO^{2}>0$.
	\end{proof}
	
	\section{Which third-order data reach \texorpdfstring{$\lone$}{l1}}
	\label{sec:elision}
	
	Three distinct mechanisms remove third-order data from $\lone$. We prove them
	separately; they have different scopes and should not be conflated.
	
	\subsection{Mechanism I: the eigenvector pairing (general)}
	\label{subsec:mechI}
	
	\begin{lemma}[Pairing lemma]\label{lem:pairing}
		In the normalization \eqref{eq:qp}, write $q=(1,q_{2})$ and
		$\bar p=(\bar p_{1},\bar p_{2})$. Then
		\begin{equation}\label{eq:pairing}
			\bar p_{2}=-\frac{i\,J_{12}}{2\wO},
			\qquad
			\bar p_{1}\,q_{2}=-\frac{i\,J_{21}}{2\wO},
			\qquad
			|q_{2}|^{2}=\frac{J_{11}^{2}+\wO^{2}}{J_{12}^{2}}\in\R .
		\end{equation}
		In particular $\bar p_{2}$ and $\bar p_{1}q_{2}$ are purely imaginary.
	\end{lemma}
	
	\begin{proof}
		From \eqref{eq:qp}, $q_{2}=-(J_{11}-i\wO)/J_{12}$ and
		$\sigma=2i\wO/(J_{11}+i\wO)$, so $\bar p_{1}=1/\sigma=(\wO-iJ_{11})/(2\wO)$
		and
		\[
		\bar p_{2}=-\frac{J_{11}-i\wO}{J_{21}}\cdot\frac{J_{11}+i\wO}{2i\wO}
		=-\frac{J_{11}^{2}+\wO^{2}}{2i\wO J_{21}}
		=\frac{J_{12}J_{21}}{2i\wO J_{21}}=-\frac{iJ_{12}}{2\wO},
		\]
		using \eqref{eq:hopfid}. Next
		\[
		\bar p_{1}q_{2}=\frac{\wO-iJ_{11}}{2\wO}\cdot\frac{i\wO-J_{11}}{J_{12}}
		=\frac{-i(J_{11}-i\wO)(J_{11}+i\wO)}{2\wO J_{12}}
		=-\frac{i(J_{11}^{2}+\wO^{2})}{2\wO J_{12}}=-\frac{iJ_{21}}{2\wO},
		\]
		again by \eqref{eq:hopfid}. Finally $q_{2}\bar q_{2}=|q_{2}|^{2}$ is as stated.
	\end{proof}
	
	\begin{proposition}[General elisions]\label{prop:generalelision}
		Under \ref{HTrace}--\ref{HDet} and in the normalization \eqref{eq:qp},
		\[
		\frac{\partial\lone}{\partial f_{1yyy}}=0,
		\qquad
		\frac{\partial\lone}{\partial f_{2xxx}}=0 .
		\]
		No graph-nullcline hypothesis, no straightening, and no assumption on the
		structure of $f$ are used.
	\end{proposition}
	
	\begin{proof}
		Neither $A$, nor $q$, $p$, nor the resolvents depend on the cubic jet, and the
		quadratic terms of \eqref{eq:kuz} involve only $B$. Hence $f_{1yyy}$ and
		$f_{2xxx}$ enter \eqref{eq:kuz} only through $\Real\ip{p}{C(q,q,\bar q)}$, and
		linearly. With $q_{1}=\bar q_{1}=1$ the corresponding monomials are
		\[
		f_{2xxx}:\quad \bar p_{2}\,q_{1}^{2}\bar q_{1}=\bar p_{2},
		\qquad
		f_{1yyy}:\quad \bar p_{1}\,q_{2}^{2}\bar q_{2}
		=(\bar p_{1}q_{2})\,|q_{2}|^{2}.
		\]
		By Lemma~\ref{lem:pairing} the first is purely imaginary, and the second is a
		purely imaginary number times a real one, hence purely imaginary. Their real
		parts vanish.
	\end{proof}
	
	\begin{remark}[Scope, and what does \emph{not} follow]\label{rem:mechIscope}
		Lemma~\ref{lem:pairing} does not make the remaining cubic pairings imaginary.
		For instance the coefficient of $f_{1xxy}$ is
		$\bar p_{1}(2q_{2}+\bar q_{2})=(3J_{11}-i\wO)(iJ_{11}-\wO)/(2\wO J_{12})$,
		whose real part vanishes only if $J_{11}=0$; similarly for $f_{2xyy}$. Thus
		Mechanism~I removes exactly the two extreme slots, and only those.
		Equivalently, on the whole Hopf set
		\begin{equation}\label{eq:mechIIvalue}
			\frac{\partial\lone}{\partial f_{1xxy}}
			=\frac{\partial\lone}{\partial f_{2xyy}}
			=-\frac{J_{11}}{2J_{12}\wO},
		\end{equation}
		which vanishes precisely when $J_{11}=0$. Mechanism~II is therefore not a
		qualitative dichotomy but a quantitative statement: the two slots contribute in
		proportion to $J_{11}=Pg'$, and straightening the nullcline is exactly what
		sets that factor to zero.
	\end{remark}
	
	\subsection{Mechanism II: straightening and parity}\label{subsec:mechII}
	
	At a Hopf point in the coordinates of Theorem~\ref{thm:normalform} one has
	$J_{11}=0$ by \eqref{eq:A0}, and the pairing acquires a full parity grading.
	
	\begin{lemma}[Eigendata in straightened coordinates]\label{lem:eigen}
		Set $\mu=\wO/P$. In the coordinates \eqref{eq:normalform}, $J_{11}=0$ and
		$J_{12}=-P$, so \eqref{eq:qp} reduces to $q=(1,-i\mu)$,
		$\bar p=(\tfrac12,\,i/(2\mu))$. Thus $q_{1},\bar p_{1}$ are real and
		$q_{2},\bar p_{2}$ purely imaginary.
	\end{lemma}
	
	\begin{proof}
		Substitute $J_{11}=0$, $J_{12}=-P$, $J_{21}=\wO^{2}/P$ into \eqref{eq:qp} and
		\eqref{eq:pairing}: $\sigma=2$, $\bar p_{1}=\tfrac12$, and
		$\bar p_{2}=-i(-P)/(2\wO)=iP/(2\wO)=i/(2\mu)$.
	\end{proof}
	
	Write $a_{jk},b_{jk}$ for $\partial_{u}^{j}\partial_{v}^{k}$ of the two
	components of \eqref{eq:normalform} at $(\xstar,0)$, so $k$ counts $v$-slots.
	
	\begin{lemma}[Parity selection]\label{lem:parity}
		For $j+k=3$ the contribution of a cubic coefficient of component $m$ to
		$\Real\ip{p}{C(q,q,\bar q)}$ vanishes unless $k$ is even for $m=1$ and odd for
		$m=2$; explicitly
		\begin{equation}\label{eq:parity}
			\Real\ip{p}{C(q,q,\bar q)}
			=\tfrac12 a_{30}+\tfrac{\mu^{2}}{2}a_{12}
			+\tfrac12 b_{21}+\tfrac{\mu^{2}}{2}b_{03}.
		\end{equation}
	\end{lemma}
	
	\begin{proof}
		The coefficient with $k$ $v$-slots multiplies a sum of products of $k$
		factors drawn from $\{q_{2},q_{2},\bar q_{2}\}$ and $j$ factors equal to
		$1$. By Lemma~\ref{lem:eigen} that multiset is $\{-i\mu,-i\mu,i\mu\}$, so
		each product is $(\pm i)^{k}\mu^{k}$ and the sum is $i^{k}\rho$ with
		$\rho$ real. Since $\bar p_{1}$ is real and $\bar p_{2}$ imaginary
		(Lemma~\ref{lem:eigen}), the contribution is real exactly in the stated
		parities. Expanding
		$C_{1}=a_{30}+i\mu a_{21}+\mu^{2}a_{12}+i\mu^{3}a_{03}$
		and likewise $C_{2}$, and taking real parts, gives \eqref{eq:parity}.
	\end{proof}
	
	\begin{theorem}[Annihilation after straightening]\label{thm:vanishing}
		In the coordinates \eqref{eq:normalform}, $\lone$ is independent of
		$a_{21},a_{03},b_{30},b_{12}$; equivalently, of
		$H_{20},H_{02},W_{30},W_{12}$, where
		$H_{jk}=\partial_{u}^{j}\partial_{v}^{k}H$ and
		$W_{jk}=\partial_{u}^{j}\partial_{v}^{k}W$ at $(\xstar,0)$.
	\end{theorem}
	
	\begin{proof}
		Cubic data enter only through $\Real\ip{p}{C(q,q,\bar q)}$; apply
		Lemma~\ref{lem:parity}. From $\tilde f_{1}=-vH$,
		\begin{equation}\label{eq:acoefs}
			\partial_{u}^{j}\partial_{v}^{k+1}\tilde f_{1}=-(k+1)H_{jk},
			\qquad \partial_{u}^{3}\tilde f_{1}|_{v=0}=0,
		\end{equation}
		so $a_{30}=0$, $a_{21}=-H_{20}$, $a_{12}=-2H_{11}$, $a_{03}=-3H_{02}$, and
		$b_{jk}=W_{jk}$.
	\end{proof}
	
	\begin{remark}
		The slots $a_{03}$ and $b_{30}$ are annihilated by both mechanisms;
		$a_{21}$ and $b_{12}$ only by Mechanism~II, which is available because
		straightening moves the equilibrium data to a coordinate frame in which
		$J_{11}=0$. Mechanism~I, by contrast, needs no change of coordinates.
	\end{remark}
	
	\section{The criticality formula}\label{sec:main}
	
	\begin{theorem}\label{thm:main}
		Let $f$ be $C^{4}$ with a graph nullcline over $I$ satisfying \ref{HGraph}, and
		let $\Pstar$ satisfy \ref{HTrace}--\ref{HDet}. In the normalization
		\eqref{eq:qp},
		\begin{equation}\label{eq:main6}
			\lone=\frac{W_{21}}{4\wO}-\frac{P W_{20}W_{11}}{4\wO^{3}}
			-\frac{W_{02}W_{11}}{4P\wO}+\frac{\wO W_{03}}{4P^{2}}
			-\frac{\wO H_{11}}{2P^{2}}+\frac{\wO H_{01}(H_{10}-W_{02})}{2P^{3}},
		\end{equation}
		and, equivalently, in terms of the trace field,
		\begin{equation}\label{eq:mainintro}
			\lone=\frac{\tau''}{4\wO}-\frac{P\,\nu''\,\tau'}{4\wO^{3}}
			-\frac{(D_{y}+P')\,\tau'}{4P\wO}+\frac{\wO D_{yy}}{4P^{2}}
			+\frac{\wO f_{1yy}D_{y}}{4P^{3}},
		\end{equation}
		all derivatives being evaluated at $\Pstar$, respectively $\xstar$.
	\end{theorem}
	
	\begin{proof}[Step 1: the cubic term]
		By Lemma~\ref{lem:parity} and \eqref{eq:acoefs},
		\begin{equation}\label{eq:pfcubic}
			\Real\ip{p}{C(q,q,\bar q)}
			=-\mu^{2}H_{11}+\tfrac12 W_{21}+\tfrac{\mu^{2}}{2}W_{03}.
		\end{equation}
	\end{proof}
	
	\begin{proof}[Step 2: resolvents]
		$\tr A_{0}=0$ and $\dt A_{0}=\wO^{2}$ give $A_{0}^{2}=-\wO^{2}I$, hence
		$A_{0}^{-1}=-A_{0}/\wO^{2}$ and, from
		$(2i\wO I-A_{0})(2i\wO I+A_{0})=-3\wO^{2}I$,
		\begin{equation}\label{eq:pfres}
			(2i\wO I-A_{0})^{-1}=-\frac{2i\wO I+A_{0}}{3\wO^{2}} .
		\end{equation}
	\end{proof}
	
	\begin{proof}[Step 3: quadratic vectors]
		With $a_{20}=0$, $a_{11}=-H_{10}$, $a_{02}=-2H_{01}$ and $b_{jk}=W_{jk}$,
		using $q_{2}+\bar q_{2}=0$, $q_{2}\bar q_{2}=\mu^{2}$, $q_{2}^{2}=-\mu^{2}$,
		\[
		B(q,\bar q)=\begin{pmatrix}-2H_{01}\mu^{2}\\ W_{20}+W_{02}\mu^{2}\end{pmatrix},
		\quad
		B(q,q)=\begin{pmatrix}2i\mu H_{10}+2H_{01}\mu^{2}\\
			W_{20}-2i\mu W_{11}-\mu^{2}W_{02}\end{pmatrix},
		\]
		whence by \eqref{eq:pfres}
		\begin{equation}\label{eq:pfh11}
			h_{11}=A_{0}^{-1}B(q,\bar q)
			=\Bigl(\frac{PW_{20}}{\wO^{2}}+\frac{W_{02}}{P},\
			\frac{2H_{01}\wO^{2}}{P^{3}}\Bigr)^{\!\top},
		\end{equation}
		and $h_{20}=-(2i\wO I+A_{0})B(q,q)/(3\wO^{2})$.
	\end{proof}
	
	\begin{proof}[Step 4: the two quadratic scalars]
		By \eqref{eq:pfres} the second resolvent vector is
		\begin{equation}\label{eq:pfh20}
			h_{20}=\Bigl(\tfrac{4H_{10}}{3P}-\tfrac{W_{02}}{3P}
			+\tfrac{PW_{20}}{3\wO^{2}}
			-i\bigl(\tfrac{4H_{01}\wO}{3P^{2}}+\tfrac{2W_{11}}{3\wO}\bigr),\ \
			-\tfrac{2H_{01}\wO^{2}}{3P^{3}}-\tfrac{4W_{11}}{3P}
			-i\tfrac{2\bigl(P^{2}W_{20}+\wO^{2}(H_{10}-W_{02})\bigr)}{3P^{2}\wO}
			\Bigr)^{\!\top}.
		\end{equation}
		Pairing $B(q,h_{11})$ and $B(\bar q,h_{20})$ with $p$ and separating real and
		imaginary parts gives
		\begin{equation}\label{eq:pfquad}
			\Real\ip{p}{B(q,h_{11})}=\Real\ip{p}{B(\bar q,h_{20})}
			=\frac{2H_{01}\wO^{4}(W_{02}-H_{10})+P^{4}W_{11}W_{20}
				+P^{2}W_{02}W_{11}\wO^{2}}{2P^{3}\wO^{2}} ,
		\end{equation}
		while the imaginary parts do \emph{not} agree:
		\[
		\Imag\ip{p}{B(q,h_{11})}
		=\frac{4H_{01}^{2}\wO^{6}+P^{6}W_{20}^{2}
			+P^{4}W_{20}\wO^{2}(H_{10}+W_{02})
			+P^{2}\wO^{4}(2H_{01}W_{11}+H_{10}W_{02})}{2P^{4}\wO^{3}},
		\]
		\[
		\begin{aligned}
			\Imag\ip{p}{B(\bar q,h_{20})}
			&=\frac{4H_{01}^{2}\wO^{6}+P^{6}W_{20}^{2}}{6P^{4}\wO^{3}}
			+\frac{P^{4}\wO^{2}\bigl(5H_{10}W_{20}+W_{02}W_{20}
				-2W_{11}^{2}\bigr)}{6P^{4}\wO^{3}}\\[2pt]
			&\quad+\frac{P^{2}\wO^{4}\bigl(10H_{01}W_{11}-2H_{10}^{2}
				+H_{10}W_{02}-2W_{02}^{2}\bigr)}{6P^{4}\wO^{3}} .
		\end{aligned}
		\]
		The coincidence in \eqref{eq:pfquad} is therefore not an accident of one
		expression being a multiple of the other; it is a property of the real parts
		alone, and it uses only Step~2 and Lemma~\ref{lem:eigen}. Since
		\eqref{eq:kuz} retains only real parts, its bracket collapses to
		\[
		\Real\ip{p}{C(q,q,\bar q)}-2\Real\ip{p}{B(q,h_{11})}
		+\Real\ip{p}{B(\bar q,h_{20})}
		=\Real\ip{p}{C(q,q,\bar q)}-\Real\ip{p}{B(q,h_{11})} .
		\]
		Inserting \eqref{eq:pfcubic} and \eqref{eq:pfquad}, substituting
		$\mu=\wO/P$ and distributing $1/(2\wO)$ over the six resulting summands
		gives \eqref{eq:main6} term by term.
	\end{proof}
	
	\begin{proof}[Step 5: the nullcline-adapted form]
		Four chain-rule identities on $\gamma$ convert \eqref{eq:main6} into
		\eqref{eq:mainintro}; none uses \ref{HTrace} or \ref{HDet}.
		By Proposition~\ref{prop:jets},
		\begin{equation}\label{eq:id5a}
			W_{11}=\tau',\qquad W_{21}=\tau'',\qquad W_{20}=\nu'' .
		\end{equation}
		Since $H(u,0)=P(u)$,
		\begin{equation}\label{eq:id5b}
			H_{10}=P'=-\bigl(f_{1xy}+g'f_{1yy}\bigr).
		\end{equation}
		Since $\partial_{v}^{2}W=f_{2yy}-g'f_{1yy}$,
		\begin{equation}\label{eq:id5c}
			W_{02}-H_{10}=f_{1xy}+f_{2yy}=D_{y},
		\end{equation}
		the terms in $g'f_{1yy}$ cancelling. Since
		$W_{03}=f_{2yyy}-g'f_{1yyy}$ and, by \eqref{eq:acoefs} with $j=k=1$,
		$H_{11}=-\tfrac12(f_{1xyy}+g'f_{1yyy})$,
		\begin{equation}\label{eq:id5d}
			\frac{\wO W_{03}}{4P^{2}}-\frac{\wO H_{11}}{2P^{2}}
			=\frac{\wO}{4P^{2}}\bigl[f_{2yyy}+f_{1xyy}\bigr]=\frac{\wO D_{yy}}{4P^{2}},
		\end{equation}
		the two occurrences of $f_{1yyy}$ cancelling. Finally
		$H_{01}=-\tfrac12 f_{1yy}$, so by \eqref{eq:id5c}
		$\wO H_{01}(H_{10}-W_{02})/(2P^{3})=\wO f_{1yy}D_{y}/(4P^{3})$.
	\end{proof}
	
	\begin{remark}[Which data determine $\lone$]\label{rem:determination}
		Formula \eqref{eq:mainintro} exhibits $\lone$ as a rational function of
		$P,P',\nu',\nu'',\tau',\tau'',D_{y},D_{yy},f_{1yy}$ at $\Pstar$, with
		$\wO^{2}=P\nu'$ and denominators powers of $P$ and $\wO$. Four of the five
		terms are built from the $2$-jet of $D$; the $2$-jet of $D$ alone does not
		suffice. What is true is that $\lone$ is determined by the $3$-jet of $f$ at
		$\Pstar$, and \eqref{eq:mainintro} identifies the combination.
	\end{remark}
	
	\begin{corollary}[Reduction of the relevant jet]\label{cor:quotient}
		Let $\mathcal J^{3}$ denote the space of $3$-jets at $\Pstar$ with fixed linear
		part $A$ satisfying \ref{HTrace}--\ref{HDet}. Then $\lone:\mathcal J^{3}\to\R$
		factors through the quotient of $\mathcal J^{3}$ by the two coordinate
		directions $\partial/\partial f_{1yyy}$ and $\partial/\partial f_{2xxx}$; that
		is, twelve of the fourteen second- and third-order coordinates survive in
		the chosen representation. No independence among those twelve is asserted:
		several of them enter $\lone$ only through fixed combinations, as
		\eqref{eq:mainintro} and \eqref{eq:decomp} show.
	\end{corollary}
	
	\begin{proof}
		Immediate from Proposition~\ref{prop:generalelision}: the two partial
		derivatives vanish identically on $\mathcal J^{3}$.
	\end{proof}
	
	\begin{corollary}[Sign criterion on the Hopf set]\label{cor:algebraic}
		With $W_{10}=\nu'(\xstar)$ and
		\begin{equation}\label{eq:Xi}
			\begin{aligned}
				\Xi:={}&P^{2}W_{10}W_{21}+PW_{10}^{2}W_{03}-2PW_{10}^{2}H_{11}\\
				&+2W_{10}^{2}H_{01}(H_{10}-W_{02})-P^{2}W_{11}W_{20}-PW_{10}W_{02}W_{11},
			\end{aligned}
		\end{equation}
		one has $\lone=\Xi/(4P\wO^{3})$, hence $\sgn\lone=\sgn\Xi$.
	\end{corollary}
	
	\begin{proof}
		Multiply \eqref{eq:main6} by $4P^{3}\wO^{3}$ and substitute
		$\wO^{2}=PW_{10}$ (Proposition~\ref{prop:hopfdata}), so
		$\wO^{4}=P^{2}W_{10}^{2}$; every term acquires a factor $P^{2}$, and dividing
		by it gives $4P\wO^{3}\lone=\Xi$. Positivity of $P$ and $\wO$ gives the sign
		statement.
	\end{proof}
	
	\begin{remark}\label{rem:algscope}
		The identity of Corollary~\ref{cor:algebraic} is not a jet identity: its proof
		substitutes $\wO^{2}=PW_{10}$, which is \ref{HDet}. It is valid on the Hopf
		set. The ten quantities in \eqref{eq:Xi} are coordinates on a jet space, not
		invariants of $f$; the statement that $\{\lone=0\}$ is algebraic refers to the
		subset of $3$-jet space cut out by $\{\tau=\nu=0,\ \Xi=0\}$ with $P,W_{10}>0$.
	\end{remark}
	
	\section{The Gause class and the third mechanism}\label{sec:gause}
	
	\begin{proposition}\label{prop:gauseclass}
		For $f$ with a graph nullcline, the following are equivalent:
		$\partial_{v}H\equiv0$; $f_{1}$ affine in $y$; $f_{1yy}\equiv0$. Every field
		with $f_{1}=x\varphi(x)-y\,p(x)$ satisfies them, with $H=p$.
	\end{proposition}
	
	\begin{proof}
		If $H=H(u)$ then $\tilde f_{1}=-vH(u)$, so $f_{1}=-(y-g(x))H(x)$ is affine in
		$y$, whence $f_{1yy}\equiv0$. Conversely $f_{1yy}\equiv0$ gives
		$\partial_{v}^{2}\tilde f_{1}\equiv0$, i.e.\
		$-2\partial_{v}H-v\partial_{v}^{2}H\equiv0$; writing $\psi=\partial_{v}H$ this
		is $v\psi_{v}+2\psi=0$, whose only solution regular at $v=0$ is $\psi\equiv0$.
		For a Gause field $\tilde f_{1}=x\varphi-(v+g)p=-vp$ since $g=x\varphi/p$.
	\end{proof}
	
	\begin{theorem}[Gause specialization]\label{thm:gause}
		If $f_{1}=x\varphi(x)-y\,p(x)$ and $f_{2}=yQ(x,y)$ then $f_{1yy}=0$ and
		\begin{equation}\label{eq:gause4}
			\lone=\frac{\tau''}{4\wO}-\frac{p\,\nu''\,\tau'}{4\wO^{3}}
			-\frac{(D_{y}+p')\,\tau'}{4p\,\wO}+\frac{\wO D_{yy}}{4p^{2}},
		\end{equation}
		with $\tau=pg'+Q+yQ_{y}$, $\nu=yQ$, $D_{y}=-p'+2Q_{y}+yQ_{yy}$,
		$D_{yy}=3Q_{yy}+yQ_{yyy}$ on $\gamma$.
	\end{theorem}
	
	\begin{proof}
		$f_{1yy}=0$ by Proposition~\ref{prop:gauseclass}; the displayed expressions
		follow from $f_{1x}=pg'$, $f_{2y}=Q+yQ_{y}$, $f_{1xy}=-p'$,
		$f_{2yy}=2Q_{y}+yQ_{yy}$, $f_{1xyy}=0$, $f_{2yyy}=3Q_{yy}+yQ_{yyy}$.
	\end{proof}
	
	\begin{corollary}[$Q_{xxx}$, revisited]\label{cor:Qxxx}
		For a Gause field, $\lone$ is independent of $Q_{xxx}(\Pstar)$. This is the
		specialization of Proposition~\ref{prop:generalelision}: since
		$f_{2}=yQ$, the only slot carrying $Q_{xxx}$ is $f_{2xxx}=yQ_{xxx}$, and
		$\partial\lone/\partial f_{2xxx}=0$ in general. No property of the Gause
		structure is used beyond the identification $f_{2xxx}=yQ_{xxx}$.
	\end{corollary}
	
	\subsection{Mechanism III: the reparametrization \texorpdfstring{$\varphi\mapsto g$}{phi to g}}
	
	The disappearance of $p'''$ is of a different nature: it is neither
	Mechanism~I nor Mechanism~II, since $\lone$ does depend on $f_{1xxx}$, and
	$f_{1xxx}=3\varphi''+x\varphi'''-y\,p'''$ does contain $p'''$.
	
	\begin{proposition}\label{prop:pppp}
		For a Gause field, express $\varphi,\varphi',\varphi'',\varphi'''$ through
		$g,g',g'',g'''$ and $p,p',p'',p'''$ by means of $g=x\varphi/p$. After this
		substitution $\lone$ is independent of $p'''(\xstar)$, and $p$ enters only
		through $p,p',p''$.
	\end{proposition}
	
	\begin{proof}
		Only $f_{1xxx}$ can carry $p'''$: the remaining slots of the prey component
		are obtained from $f_{1}=x\varphi-yp$ by at least one differentiation in
		$y$, which lowers the order in $p$. After the substitution
		$x\varphi=p\,g$ one has $f_{1}=p\,g-y\,p$, hence
		\begin{equation}\label{eq:pxxx}
			f_{1xxx}=\bigl(g-y\bigr)p'''+3g'p''+3g''p'+p\,g''' ,
		\end{equation}
		so that
		\[
		\frac{\partial f_{1xxx}}{\partial p'''}=g(x)-y .
		\]
		At an equilibrium on the nullcline $\ystar=g(\xstar)$, whence
		\[
		\left.\frac{\partial f_{1xxx}}{\partial p'''}\right|_{\Pstar}=0 .
		\]
		The cancellation is therefore not a feature of any particular way of writing
		$\lone$: the coefficient of $p'''$ in the only slot that could carry it is
		$g-y$, and that quantity vanishes at $\Pstar$ precisely because $\Pstar$ lies
		on $\gamma$. Since $\lone$ depends on the prey component only through its
		$3$-jet at $\Pstar$, and $p'''$ has already been removed from $f_{1xxx}$
		there, $p'''$ cannot occur in $\lone$. By \eqref{eq:pxxx} the surviving
		$p$-derivatives in that slot are $p''$, $p'$ and $p$, consistent with
		\eqref{eq:gause4}, where the highest one appearing is $p''$ inside
		$\tau''$.
	\end{proof}
	
	\begin{remark}[Three mechanisms, three scopes]\label{rem:threemech}
		It is worth isolating what each mechanism does.
		\begin{enumerate}[label=\textup{(\Roman*)},nosep]
			\item \emph{Pairing.} $f_{1yyy}$ and $f_{2xxx}$ never reach $\lone$, for any
			planar field at a Hopf point, in any coordinates, with or without a nullcline
			(Proposition~\ref{prop:generalelision}).
			\item \emph{Straightening and parity.} In the coordinates
			\eqref{eq:normalform}, where $J_{11}=0$, two further slots are annihilated
			(Theorem~\ref{thm:vanishing}); by \eqref{eq:mechIIvalue} their contribution is
			proportional to $J_{11}=Pg'$. This mechanism requires the graph hypothesis.
			\item \emph{Reparametrization.} Within the Gause class, writing the prey jet in
			terms of $g$ rather than $\varphi$ removes $p'''$
			(Proposition~\ref{prop:pppp}). This is a change of variables in the jet, not a
			cancellation in the pairing.
		\end{enumerate}
		Earlier accounts attributed (I) to the Gause structure; it is general, and the
		Gause calculation merely makes it visible. Only (III) is genuinely
		Gause-specific.
	\end{remark}
	
	\section{Predator-dependent responses}\label{sec:predator}
	
	Consider, on a domain in $\{y>0\}$,
	\begin{equation}\label{eq:sys}
		\dot x=x\varphi(x)-y\,\Phi(x,y),\qquad \dot y=y\,Q(x,y),
	\end{equation}
	with $\Phi\in C^{4}$, so that $f_{1y}=-(\Phi+y\Phi_{y})$ and
	\begin{equation}\label{eq:carriers}
		f_{1yy}=-2\Phi_{y}-y\Phi_{yy},\quad
		f_{1xyy}=-2\Phi_{xy}-y\Phi_{xyy},\quad
		f_{1yyy}=-3\Phi_{yy}-y\Phi_{yyy}.
	\end{equation}
	
	\subsection{What the carriers characterize}
	
	\begin{proposition}\label{prop:carriers}
		Let $V\subseteq\{y>0\}$ be open with connected vertical sections.
		\begin{enumerate}[label=\textup{(\alph*)},nosep]
			\item $\Phi_{y}\equiv0$ on $V$ implies $f_{1yy}\equiv f_{1xyy}\equiv
			f_{1yyy}\equiv0$.
			\item $f_{1yy}\equiv0$ on $V$ if and only if $\Phi(x,y)=A(x)/y+B(x)$ there.
			\item The converse of \textup{(a)} holds if and only if $A\equiv0$; a
			sufficient condition is that $\Phi$ be bounded near $y=0$, or extend
			continuously to $y=0$.
		\end{enumerate}
	\end{proposition}
	
	\begin{proof}
		(a) is immediate. (b) $f_{1yy}=-\partial_{y}^{2}(y\Phi)$, so $f_{1yy}\equiv0$
		iff $y\Phi$ is affine in $y$ on each vertical section, iff $\Phi=A/y+B$.
		(c) Then $\Phi_{y}=-A/y^{2}$, zero iff $A\equiv0$; boundedness near $y=0$
		forces $y\Phi\to0$, hence $A=0$.
	\end{proof}
	
	\begin{example}\label{ex:counter}
		With $A\not\equiv0$, $B>0$ and $\Phi=A(x)/y+B(x)$, one has
		$f_{1}=x\varphi-A-B y$ affine in $y$, so all three carriers vanish
		identically, while $\Phi_{y}=-A/y^{2}\not\equiv0$; and $f_{1y}=-B<0$, so
		\ref{HGraph} holds. Hence the assertion that the three carriers vanish
		identically if and only if $\Phi_{y}\equiv0$ fails on the open quadrant; it
		becomes valid under the boundedness hypothesis of
		Proposition~\ref{prop:carriers}(c), which holds for Beddington--DeAngelis
		\cite{Beddington1975,DeAngelis1975} and Crowley--Martin \cite{CrowleyMartin1989}.
	\end{example}
	
	\subsection{The affine truncation}
	
	\begin{definition}[Affine truncation]\label{def:aff}
		Let $\pi$ act on the $3$-jet $j^{3}_{\Pstar}f$ by replacing the coefficients
		\[
		f_{1yy}\mapsto0,\qquad f_{1xyy}\mapsto0,
		\]
		and leaving every other coefficient unchanged; in particular $\pi$ preserves
		the whole $1$-jet, hence $A$ and $\wO$. Since $\lone$ depends only on the
		$3$-jet, we set
		\[
		\lone^{\mathrm{aff}}:=\lone\bigl(\pi(j^{3}_{\Pstar}f)\bigr).
		\]
		A field realizing $\pi(j^{3}_{\Pstar}f)$ is its Taylor polynomial of degree
		three at $\Pstar$.
	\end{definition}
	
	\begin{remark}\label{rem:redundant}
		Imposing in addition $f_{1yyy}\mapsto0$ yields the same value of
		$\lone^{\mathrm{aff}}$, by Proposition~\ref{prop:generalelision}. The
		truncation therefore removes exactly the two nonlinear predator-dependent slots
		that can affect $\lone$.
	\end{remark}
	
	\begin{remark}[Why not ``freezing''\,]\label{rem:freezing}
		Replacing $\Phi(x,y)$ by $\Phi(x,\ystar)$ is a different operation. It
		preserves the equilibrium and, since $f_{1x}$ is unchanged at $\Pstar$, the
		trace; but it replaces $f_{1y}=-(\Phi+y\Phi_{y})$ by $-\Phi$, altering
		$J_{12}$, hence $\dt A$ and $\wO$. For the Beddington--DeAngelis point of
		\S\ref{sec:validation} it moves $J_{12}$ from $-0.332457$ to $-0.375$ and
		$\wO$ from $0.476144$ to $0.505839$. The comparison
		$\lone-\lone^{\mathrm{aff}}$ is meaningful precisely because
		Definition~\ref{def:aff} keeps $A$ and $\wO$ fixed.
	\end{remark}
	
	\subsection{The decomposition}
	
	\begin{theorem}\label{thm:decomp}
		Let $f$ be $C^{4}$ near $\Pstar$ with $f(\Pstar)=0$ and $A=Df(\Pstar)$, and
		assume \ref{HTrace}--\ref{HDet}; then $J_{12}\neq0$ and
		$J_{11}^{2}+\wO^{2}=-J_{12}J_{21}>0$. Adopt the normalization \eqref{eq:qp} and
		let $\lone^{\mathrm{aff}}$ be as in Definition~\ref{def:aff}. With
		$D=f_{1x}+f_{2y}$, $\nabla D=(D_{x},D_{y})$ and
		$\partial_{x}f=Ae_{1}=(J_{11},J_{21})^{\!\top}$, one has
		\begin{equation}\label{eq:decomp}
			\lone-\lone^{\mathrm{aff}}=\Delta
			=\Lambda\Bigl(f_{1xyy}
			+\frac{\ip{\partial_{x}f}{\nabla D}}{\wO^{2}}\,f_{1yy}\Bigr),
		\end{equation}
		where
		\begin{equation}\label{eq:Lambda}
			\Lambda=\frac{J_{11}^{2}+\wO^{2}}{4\wO J_{12}^{2}}
			=-\frac{J_{21}}{4\wO J_{12}}>0 .
		\end{equation}
		Moreover $\lone$ has degree at most one in each of $f_{1yy}$ and $f_{1xyy}$ and
		degree zero in $f_{1yyy}$; the positivity of $\Lambda$ follows from
		\ref{HTrace}--\ref{HDet} alone, and no orientation hypothesis such as
		$J_{12}<0$, $J_{21}>0$ is used.
	\end{theorem}
	
	\begin{proof}
		Degree zero in $f_{1yyy}$ is Proposition~\ref{prop:generalelision}. The
		eigendata \eqref{eq:qp} and the resolvents depend only on $A$, so $f_{1xyy}$
		enters \eqref{eq:kuz} only through $C$, hence linearly. The coefficient
		$f_{1yy}$ enters through $B$, hence through $h_{11},h_{20}$ and the outer
		forms, a priori quadratically; its quadratic part is
		\[
		\bar p_{1}\Bigl[\bar q_{2}\bigl(R_{20}q_{2}^{2}e_{1}\bigr)_{2}
		+2q_{2}\bigl(R_{11}q_{2}\bar q_{2}e_{1}\bigr)_{2}\Bigr],
		\qquad R_{20}=(2i\wO I-A)^{-1},\ R_{11}=-A^{-1},
		\]
		Substituting \eqref{eq:qp} and using $q_{2}\bar q_{2}=(J_{11}^{2}+\wO^{2})
		/J_{12}^{2}$ together with \eqref{eq:hopfid}, this quadratic pairing
		evaluates to
		\begin{equation}\label{eq:quadcancel}
			-\,\frac{5i\,(J_{11}^{2}+\wO^{2})^{3}}{6\,J_{12}^{4}\,\wO^{3}},
		\end{equation}
		which is purely imaginary because $J_{11},J_{12},\wO$ are real. Its real part
		vanishes, so no $f_{1yy}^{2}$ term survives in \eqref{eq:kuz}; the same
		computation gives degree zero in $f_{1yy}^{3}$. Hence
		$\Delta=\kappa_{02}f_{1yy}+\kappa_{12}f_{1xyy}$, and this is an exact
		statement about $\lone$, not the outcome of a simplification.
		
		To evaluate the coefficients, substitute \eqref{eq:qp} into \eqref{eq:kuz},
		form the difference against $\pi(j^{3}_{\Pstar}f)$, and eliminate
		$J_{22}=-J_{11}$ and $J_{21}=-(J_{11}^{2}+\wO^{2})/J_{12}$ by \ref{HTrace} and
		\eqref{eq:hopfid}. One obtains $\kappa_{12}=\Lambda$ and
		\[
		\kappa_{02}
		=\frac{(J_{11}^{2}+\wO^{2})\bigl[J_{11}J_{12}D_{x}
			-(J_{11}^{2}+\wO^{2})D_{y}\bigr]}{4J_{12}^{3}\wO^{3}}
		=\frac{\Lambda}{\wO^{2}}\bigl(J_{11}D_{x}+J_{21}D_{y}\bigr),
		\]
		the second equality by \eqref{eq:hopfid}. Since $D=f_{1x}+f_{2y}$ gives
		$D_{x}=f_{1xx}+f_{2xy}$ and $D_{y}=f_{1xy}+f_{2yy}$, the bracket is
		$\ip{Ae_{1}}{\nabla D}$.
		
		Positivity: $J_{11}^{2}+\wO^{2}>0$ because $\wO>0$, and $4\wO J_{12}^{2}>0$
		because $J_{12}$ is real and nonzero. The second form in \eqref{eq:Lambda}
		follows from \eqref{eq:hopfid}, which also shows $J_{21}/J_{12}<0$; that
		inequality is a consequence of \ref{HTrace}--\ref{HDet}, not an extra
		hypothesis.
	\end{proof}
	
	\begin{remark}[$\Lambda>0$ does not sign $\Delta$]\label{rem:lambdasign}
		Positivity of $\Lambda$ fixes no sign for $\Delta$. The bracket in
		\eqref{eq:decomp} involves $f_{1xyy}$, $f_{1yy}$ and
		$\ip{\partial_{x}f}{\nabla D}$, none of which is signed by
		\ref{HTrace}--\ref{HDet}; $\Lambda$ only guarantees that $\Delta$ and the
		bracket have the same sign. In particular $\Delta$ may increase or decrease
		$\lone$, and both possibilities are compatible with the hypotheses.
	\end{remark}
	
	\begin{remark}[Normalization weight]\label{rem:covariance}
		Under $q\mapsto cq$ with the corresponding adjoint normalization,
		\[
		\lone\mapsto|c|^{2}\lone,\quad
		\lone^{\mathrm{aff}}\mapsto|c|^{2}\lone^{\mathrm{aff}},\quad
		\Delta\mapsto|c|^{2}\Delta,\quad
		\Lambda\mapsto|c|^{2}\Lambda,
		\]
		by Lemma~\ref{lem:normalization} and because $\pi$ acts on the jet, not on the
		eigenvector. The bracket
		$f_{1xyy}+\wO^{-2}\ip{\partial_{x}f}{\nabla D}f_{1yy}$ is built from jet data
		and entries of $A$ only, hence is normalization-independent. Theorem
		In \eqref{eq:decomp}, therefore, the entire dependence on the choice of $q$
		sits in $\Lambda$, which is homogeneous of degree one in the weight
		$|c|^{2}$, while the bracket is unaffected. We call neither side invariant:
		$\lone$, $\lone^{\mathrm{aff}}$, $\Delta$ and $\Lambda$ are all
		normalization-dependent, and only their signs and mutual ratios are not. In the convention $q=(J_{12},i\wO-J_{11})$ one has
		$\Lambda=(J_{11}^{2}+\wO^{2})/(4\wO)$. The signs of $\lone$ and of $\Delta$,
		and the ratio $\Delta/\lone$, are independent of the convention.
	\end{remark}
	
	\begin{corollary}[Reduction of $\Delta$]\label{cor:deltajet}
		$\Delta$ depends on the $3$-jet only through $f_{1yy}$, $f_{1xyy}$ and the four
		second-order coefficients $f_{1xx},f_{2xy},f_{1xy},f_{2yy}$ entering
		$\nabla D$. It is independent of $f_{1xxx},f_{1xxy},f_{1yyy},f_{2xx}$ and of
		the entire cubic jet of $f_{2}$.
	\end{corollary}
	
	\begin{proof}
		Read off \eqref{eq:decomp}: $\Lambda$ depends only on $A$ and $\wO$, and the
		bracket only on the six listed coefficients.
	\end{proof}
	
	\subsection{The bridge to the localization theorem}
	
	\begin{corollary}[Decomposition of the pairing]\label{cor:bridge}
		Assume in addition that $f$ has a graph nullcline with $P=-f_{1y}(\Pstar)>0$,
		and set $\tau(x)=D(x,g(x))$. Then
		\begin{equation}\label{eq:bridge}
			\ip{\partial_{x}f}{\nabla D}=J_{11}\,\tau'+\frac{\wO^{2}}{P}\,D_{y},
		\end{equation}
		and consequently, using $J_{11}=P\,g'$ from Lemma~\ref{lem:traceid},
		\begin{equation}\label{eq:deltaloc2}
			\Delta=\Lambda\Bigl[f_{1xyy}
			+\Bigl(\frac{P\,g'}{\wO^{2}}\,\tau'+\frac{D_{y}}{P}\Bigr)f_{1yy}\Bigr].
		\end{equation}
	\end{corollary}
	
	\begin{proof}
		Since $\tau(x)=D(x,g(x))$ we have $\tau'=D_{x}+g'D_{y}$, so
		$D_{x}=\tau'-g'D_{y}$ and
		\[
		\ip{\partial_{x}f}{\nabla D}=J_{11}D_{x}+J_{21}D_{y}
		=J_{11}\tau'+\bigl(J_{21}-J_{11}g'\bigr)D_{y}.
		\]
		By Lemma~\ref{lem:traceid}, $J_{11}=Pg'$ and $J_{12}=-P$, so
		$J_{11}g'=J_{11}^{2}/P=-J_{11}^{2}/J_{12}$, while \eqref{eq:hopfid} gives
		$J_{21}=-(J_{11}^{2}+\wO^{2})/J_{12}$. Hence
		\[
		J_{21}-J_{11}g'=-\frac{J_{11}^{2}+\wO^{2}}{J_{12}}+\frac{J_{11}^{2}}{J_{12}}
		=-\frac{\wO^{2}}{J_{12}}=\frac{\wO^{2}}{P},
		\]
		which is \eqref{eq:bridge}. Substituting into \eqref{eq:decomp} and using
		$J_{11}=Pg'$ gives \eqref{eq:deltaloc2}.
	\end{proof}
	
	\begin{remark}[The same object at two orders]\label{rem:twoorders}
		Formula \eqref{eq:deltaloc2} exhibits $J_{11}=Pg'$ twice over. At first order
		it is the quantity whose positivity is forced by $\tr A=0$ together with
		$f_{2y}<0$, and which therefore confines $\Pstar$ to an ascending branch
		(Theorem~\ref{thm:localization}, \cite{Companion}). At third order it
		is the weight multiplying the tangential derivative $\tau'$ of the trace field
		inside the predator-dependent correction. The transverse contribution carries
		instead the weight $1/P$, with $P>0$ the marginal predation rate. The two
		statements are successive orders of one local analysis.
	\end{remark}
	
	\begin{remark}[What $\lone^{\mathrm{aff}}$ is, and is not]\label{rem:affmeaning}
		$\lone^{\mathrm{aff}}$ is the first Lyapunov coefficient of the $y$-affine
		truncation of the $3$-jet at fixed linear part. It is not ``the nullcline
		contribution'': by \eqref{eq:gause4} it still contains $\tau''$, $\nu''$ and
		$D_{yy}$. What $\Delta$ isolates is exactly the contribution of the two
		predator-dependent curvature slots $f_{1yy}$ and $f_{1xyy}$, which by
		Proposition~\ref{prop:generalelision} and \eqref{eq:carriers} are the only
		channels through which the predator dependence of $\Phi$ can reach $\lone$.
	\end{remark}
	
	\begin{remark}[Two channels]\label{rem:twochannels}
		The bracket in \eqref{eq:decomp} has two terms of different type. The first,
		$f_{1xyy}$, involves no trace-field data. The second couples $f_{1yy}$ to
		$\wO^{-2}\ip{\partial_{x}f}{\nabla D}$. It would therefore be inaccurate to say
		that predator dependence acts only through the trace field. The accurate
		statement is that it acts through exactly two scalars, and that the weight
		attached to the second admits the decomposition \eqref{eq:bridge}. Within the
		equilibrium set the Hopf condition is $D=0$, so $\nabla D$ is conormal to the
		Hopf locus and $\ip{\partial_{x}f}{\nabla D}$ is the derivative of $D$ along
		$Ae_{1}$; it vanishes exactly when $Ae_{1}$ is tangent to $\{D=0\}$, and then
		the second channel closes.
	\end{remark}
	
	\section{Specializations and validation}\label{sec:validation}
	
	We distinguish three activities. An analytic proof establishes a statement
	for all fields satisfying the hypotheses; all theorems above are of this
	kind. A symbolic verification confirms an algebraic identity between
	expressions in indeterminates, and is a proof of that identity; we invoke it
	once, in Step~4 of Theorem~\ref{thm:main}. A numerical validation evaluates
	two formulas on one field: it can refute but not prove. This section is of
	the second and third kinds, and it is deliberately restricted to two
	representative realizations, chosen for complementary purposes. The Bazykin
	model provides the dynamical illustration of the reduced coefficient, while a
	single Beddington--DeAngelis Hopf point provides a nondegenerate realization
	of the decomposition $\lone=\lone^{\mathrm{aff}}+\Delta$.
	
	\subsection{A reparametrized Bazykin model}\label{subsec:bazykin}
	Consider, on the open first quadrant, the predator--prey system of
	Bazykin \cite{Banerjee2017,Bazykin1998} in the dimensionless form
	\begin{equation}\label{eq:bazykin}
		\dot x=\rho\,x\Bigl(1-\frac{x}{k}\Bigr)-\frac{\beta\,x y}{1+x},
		\qquad
		\dot y=\mu\Bigl(\frac{\alpha\,x y}{1+x}-y-\xi\,y^{2}\Bigr).
	\end{equation}
	The prey component is affine in $y$, so \eqref{eq:bazykin} is a Gause field
	with functional response $p(x)=\beta x/(1+x)$, and
	Theorem~\ref{thm:gause} applies. Writing $k_{0}$ for the prey density at the
	coexistence equilibrium, the remaining parameters are fixed by
	\begin{equation}\label{eq:bazrepar}
		k=2(1+k_{0}),\qquad
		\alpha=\frac{(1+k_{0})\bigl(2\beta+(2+k_{0})\xi\rho\bigr)}{2k_{0}\beta},
		\qquad
		\mu_{0}=\frac{k_{0}\beta}{(1+k_{0})^{2}(2+k_{0})\,\xi},
	\end{equation}
	and the equilibrium is
	$\Pstar=\bigl(k_{0},(2+k_{0})\rho/(2\beta)\bigr)$, at which $\tr A=0$ when
	$\mu=\mu_{0}$.
	
	With $k_{0}=4$, $\beta=\tfrac12$, $\rho=\tfrac15$ and $\xi=\tfrac1{30}$,
	\eqref{eq:bazrepar} gives $k=10$, $\alpha=\tfrac{13}{10}$,
	$\mu_{0}=\tfrac25$ and $\Pstar=(4,\tfrac65)$, where
	\[
	\tr A=0,\qquad
	\dt A=\frac{152}{15625},\qquad
	\wO=\frac{2\sqrt{38}}{125}.
	\]
	The prey nullcline is the parabola
	$g(x)=\rho(1-x/k)(1+x)/\beta=\tfrac25+\tfrac9{25}x-\tfrac1{25}x^{2}$, whose
	only critical point is $x_{\mathrm v}=(k-1)/2=\tfrac92$. At $\Pstar$,
	\begin{equation}\label{eq:bazloc}
		g'(\xstar)=\tfrac1{25}>0,
		\qquad
		P=-f_{1y}(\Pstar)=\tfrac25,
		\qquad
		P\,g'(\xstar)=\tfrac2{125}=J_{11},
	\end{equation}
	so $\Pstar$ lies on the ascending branch, strictly to the left of
	$x_{\mathrm v}$, as Theorem~\ref{thm:localization} requires.
	
	The remaining data of \eqref{eq:mainintro} at $\Pstar$ are
	\[
	P'=\tfrac1{50},\quad
	\tau'=-\tfrac{43}{3750},\quad
	\tau''=-\tfrac{91}{9375},\quad
	\nu''=-\tfrac{332}{46875},\quad
	D_{y}=-\tfrac{7}{150},\quad
	D_{yy}=0,
	\]
	together with $f_{1yy}=0$. Substituting into \eqref{eq:mainintro} gives the
	exact value
	\begin{equation}\label{eq:bazl1}
		\lone=-\frac{41\sqrt{38}}{7220}=-0.035005675086\ldots<0,
	\end{equation}
	so the crossing is supercritical. Evaluating \eqref{eq:kuz} directly at the
	same point and in the normalization of Definition~\ref{def:norm} returns the
	same number; the difference simplifies to zero exactly.
	
	\begin{remark}[What this example does and does not test]\label{rem:bazscope}
		Since $f_{1}$ is affine in $y$ we have $f_{1yy}=f_{1xyy}=0$, hence
		$\Delta=0$ by Theorem~\ref{thm:decomp}. The example therefore validates
		\eqref{eq:mainintro}, the localization identity \eqref{eq:bazloc} and the
		vanishing of $\Delta$ in the predator-independent case, but it is
		degenerate for the decomposition itself: a nonzero $\Delta$ requires a
		predator-dependent response. One such point is recorded in
		\S\ref{subsec:predep}.
	\end{remark}
	
	\subsection{The bifurcating cycle}\label{subsec:cycle}
	Lowering $\mu$ from $\mu_{0}=\tfrac25$ to $\mu=\tfrac15$, with the remaining
	parameters unchanged, destabilizes $\Pstar$ and, $\lone$ being negative, a
	stable limit cycle is present. Figure~\ref{fig:cycle} shows it.
	
	\begin{figure}[ht]
		\centering
		\includegraphics[width=0.66\textwidth]{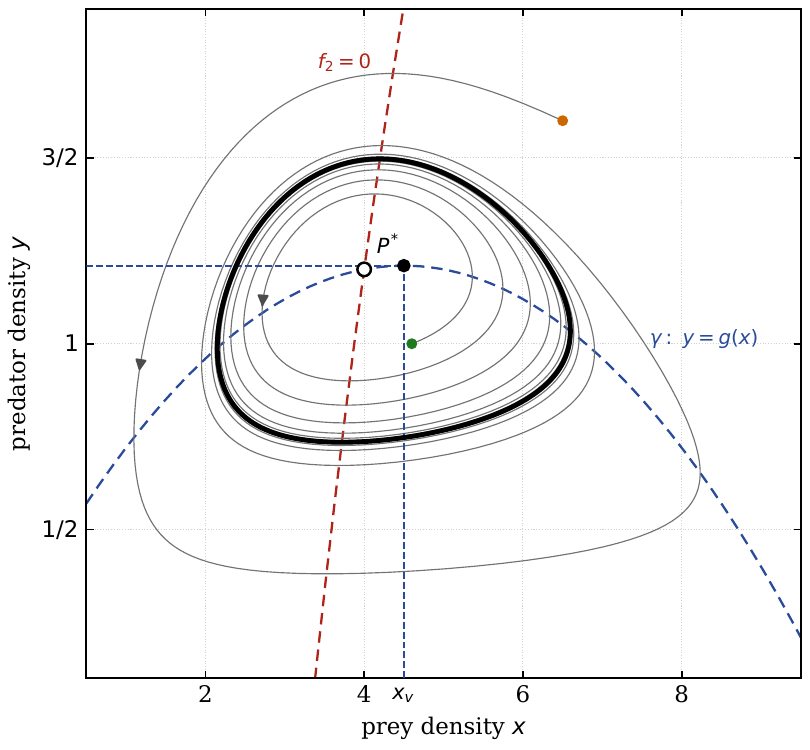}
		\caption{The system \eqref{eq:bazykin} at $k_{0}=4$, $\beta=\tfrac12$,
			$\rho=\tfrac15$, $\xi=\tfrac1{30}$ and $\mu=\tfrac15<\mu_{0}$. Dashed:
			the prey nullcline $\gamma$ (blue, with its critical point
			$x_{\mathrm v}=\tfrac92$ marked) and the nontrivial predator nullcline
			(red). The equilibrium $\Pstar=(4,\tfrac65)$ is an unstable focus (open
			circle) and sits on the ascending branch of $\gamma$, left of
			$x_{\mathrm v}$, in accordance with Theorem~\ref{thm:localization}. Two
			orbits, one from outside and one from inside, approach the limit cycle
			(thick). Since $\lone<0$ at the Hopf point by \eqref{eq:bazl1}, the local
			Hopf bifurcation is supercritical; the stable cycle shown is the
			corresponding numerical realization.}
		\label{fig:cycle}
	\end{figure}
	
	\subsection{A predator-dependent point}\label{subsec:predep}
	The decomposition \eqref{eq:decomp} is not visible in \eqref{eq:bazykin}, so
	we record one point of a Beddington--DeAngelis system
	\cite{Beddington1975,DeAngelis1975}, with
	$\Phi=ax/(1+bx+cy)$ and $\rho=a=1$, $b=\tfrac14$, $c=\tfrac3{25}$,
	$\gamma=\tfrac45$, $d=\tfrac3{10}$. The Hopf point is located by a
	high-precision nonlinear solve of $\{f_{1}=0,\ f_{2}=0,\ \tr A=0\}$ in
	$(\xstar,\ystar,k)$, only $k$ being substituted into the field, so that the
	Jacobian is that of the field. At $k^{*}\approx8.66359$ one finds, at $50$
	significant digits and in the normalization of Definition~\ref{def:norm},
	\[
	\wO=0.476144193409,\quad
	\lone=-0.159728456856,\quad
	\lone^{\mathrm{aff}}=-0.141426790554,
	\]
	hence $\Delta=-0.018301666302$ with $\Lambda=1.082483243$, and
	$g'(\xstar)=0.1023727006>0$. The direct difference
	$\lone-\lone^{\mathrm{aff}}$ and the closed form \eqref{eq:decomp} agree to
	below $10^{-40}$, as do the two sides of \eqref{eq:bridge}; the residual of
	$J_{11}=Pg'$ is limited by the precision of the located root rather than by
	the arithmetic and is reported against a separate, looser tolerance.
	
	\begin{remark}[Reading the $11.5\%$]\label{rem:38}
		At this point $|\lone^{\mathrm{aff}}|=0.8854\,|\lone|$, so the affine
		truncation reproduces the criticality sign while changing the coefficient
		governing the leading-order amplitude scaling by about $11.5\%$. This is not a
		statement about an $11.5\%$ error in a physical amplitude: in the normal
		form $\dot z=(\alpha+i\wO)z+c_{1}z|z|^{2}+O(|z|^{4})$ one has
		$r^{2}\sim-\alpha/\Real(c_{1})$, so $r\propto|\lone|^{-1/2}$ at fixed
		$\alpha$, and the leading-order radius changes by the factor
		$(0.8854)^{-1/2}=1.063$. Both ratios are independent of the eigenvector
		normalization.
	\end{remark}
	
	\section{Discussion}\label{sec:discussion}
	
	\subsection{Two orders of one organization}
	At first order the identity $J_{11}=Pg'$ converts the spectral condition
	$\tr A=0$ into the geometric restriction $g'(\xstar)>0$, so that the critical
	points of the prey nullcline are spectral barriers and the Hopf locus lies on
	ascending branches \cite{Companion}. At third order the same quantity
	$J_{11}=Pg'$ reappears, in \eqref{eq:deltaloc2}, as the weight multiplying the
	tangential derivative $\tau'$ of the trace field inside the predator-dependent
	correction to $\lone$, and again in \eqref{eq:mechIIvalue} as the factor
	controlling the two slots that straightening annihilates. The two results are
	successive local orders of one geometric organization of the Hopf bifurcation,
	and the present paper does not replace or duplicate the localization theorem:
	it answers the question that theorem leaves open, namely how the bifurcation
	unfolds once its location is constrained.
	
	We repeat the qualification of Remark~\ref{rem:affmeaning}:
	$\lone^{\mathrm{aff}}$ is not ``the nullcline part'' of $\lone$. It is the
	first Lyapunov coefficient of the $y$-affine jet truncation at fixed linear
	part, and it still contains $\tau''$, $\nu''$ and $D_{yy}$. What $\Delta$
	isolates is the contribution of $f_{1yy}$ and $f_{1xyy}$, the only two slots
	through which predator dependence can act.
	
	\subsection{Which third-order data matter}
	Corollaries~\ref{cor:quotient} and \ref{cor:deltajet} record that criticality
	does not use all of the $3$-jet. Two slots, $f_{1yyy}$ and $f_{2xxx}$, are
	irrelevant for every planar field at a Hopf point; two further slots are
	annihilated once the nullcline is straightened; and within the Gause class the
	reparametrization $\varphi\mapsto g$ eliminates $p'''$. The first phenomenon is
	a reality property of the eigenvector pairing, the second a parity grading
	available only after straightening, the third a change of variables in the jet.
	They are independent, and the earlier attribution of the first to the Gause
	structure was too narrow.
	
	\subsection{Bautin candidates}
	Along an ascending branch the Hopf set is $\{\nu=\tau=0,\ P\nu'>0\}$, and
	\eqref{eq:mainintro} gives $\lone$ there. A point of the Hopf set with
	$\lone=0$ is a \emph{Bautin candidate}. If $\lone$ takes opposite signs at two
	points of the same connected component, a Bautin candidate lies strictly
	between them, and $\Xi$ of \eqref{eq:Xi} changes sign there; this follows from
	the intermediate value theorem, the jet data depending continuously on
	$\xstar$. A Bautin candidate need not be a Bautin point: a nondegenerate
	generalized Hopf bifurcation requires in addition $\ell_{2}\neq0$ and a
	transversality condition on the map from parameters to $(\tr A,\lone)$
	\cite[\S8.3]{Kuznetsov2004}. Neither follows from the present hypotheses; the
	first requires the $5$-jet. We therefore do not write ``Bautin point'' for a
	point satisfying only $\tau=\lone=0$, and we assert no fold of limit cycles.
	
	\subsection{Limits}
	The graph hypothesis is essential: at a vertical tangency of the first
	nullcline, $P\to0$ and \eqref{eq:mainintro} is singular. The parity mechanism
	of \S\ref{subsec:mechII} uses $J_{11}=0$ in the straightened frame, hence is
	available only at points satisfying \ref{HTrace}. In dimension $n>2$ a
	codimension-one nullcline can be straightened and the analogue of $\tau$ is the
	restriction of the trace to it, but the transverse block no longer reduces to a
	single $2\times2$ matrix, so the eigendata of Lemma~\ref{lem:eigen} are not
	available and the parity grading of Lemma~\ref{lem:parity} does not extend
	verbatim. Lemma~\ref{lem:pairing} is likewise a statement about a planar
	pairing; its analogue in higher dimension would require a separate analysis.
	Finally, \cite{Companion} also localizes the Neimark--Sacker locus of the
	forward Euler discretization; the corresponding criticality analysis is not
	attempted here.
	
	\subsection*{Supplementary material}
	
	A self-contained Mathematica notebook in Wolfram Language is provided as
	ancillary material. It is organized in foldable sections following the
	argument of this paper, and separates computation, verification and
	presentation. The notebook distinguishes exact symbolic checks from
	high-precision numerical validations, with numerical residuals reported
	where appropriate. It evaluates \eqref{eq:kuz}, \eqref{eq:mainintro},
	\eqref{eq:decomp}, \eqref{eq:bridge} and \eqref{eq:gause4}, carries out the
	symbolic verification used in Step~4 of Theorem~\ref{thm:main}, confirms
	\eqref{eq:mechIIvalue} and the elisions of
	Proposition~\ref{prop:generalelision}, reproduces the exact Bazykin values
	reported in \S\ref{sec:validation}, and reproduces
	Figure~\ref{fig:cycle}. The notebook also verifies the bridge identity at the
	Beddington--DeAngelis Hopf point and evaluates the decomposition
	$\ell_1=\ell_1^{\mathrm{aff}}+\Delta$, including the corresponding closed
	expression for $\Delta$. Finally, it contains an exact parameter assignment
	for which $\ell_1=0$.
		
	\subsection*{Data availability}
	
	No datasets were generated or analysed during this study. The
	\textsc{Wolfram Language} verification notebook described in the
	Supplementary material is provided with this article as ancillary material.
	The notebook is self-contained and can be executed directly in Mathematica.
	
	\section*{Acknowledgements}
	E. Chan-L\'opez acknowledges support from SECIHTI through the
	``Estancias Posdoctorales por M\'exico'' program (CVU 422090). The author thanks V. Castellanos for helpful comments and suggestions.
	
	\section*{Ethics declarations}
	\subsection*{Conflict of interest}
	The author declares no conflict of interest.
	

\end{document}